\documentclass[reqno,a4paper,12pt]{amsart}

\usepackage[T1]{fontenc}
\usepackage[utf8]{inputenc}

\usepackage{geometry}
\usepackage{mathtools, amsfonts, amssymb, amsthm}

\theoremstyle{plain}
\newtheorem{theorem}{Theorem}[section]
\newtheorem{corollary}[theorem]{Corollary}
\newtheorem{lemma}[theorem]{Lemma}
\newtheorem{proposition}[theorem]{Proposition}

\theoremstyle{definition}
\newtheorem{definition}[theorem]{Definition}
\newtheorem{example}[theorem]{Example}
\newtheorem*{remark}{Remark}

\newcommand{\inner}[2]{\langle #1, #2 \rangle}

\newcommand{\KK}{\mathbb K}
\newcommand{\NN}{\mathbb N}
\newcommand{\RR}{\mathbb R}

\newcommand{\M}{\mathcal M}
\newcommand{\N}{\mathcal N}
\newcommand{\eps}{\varepsilon}
\renewcommand{\emptyset}{\varnothing}
\renewcommand{\leq}{\leqslant}
\renewcommand{\geq}{\geqslant}

\title[Essential Submodules of Hilbert $C^*$-Modules]{On Essential and Topologically Essential \\ Submodules of Hilbert $C^*$-Modules}
\author{K. Kartvelishvili}
\date{}
\address{Moscow Center for Fundamental and Applied Mathematics, Moscow State University,
	Leninskie Gory 1, Moscow, 
	119991, Russia}

\email{kirill.kartvelishvili@math.msu.ru}

\begin{document}
	
	\begin{abstract}
		We study two notions of largeness for closed submodules of Hilbert $C^*$-modules: essentiality and topological essentiality.
		While the analogous properties are known to be equivalent for closed two-sided ideals of $C^*$-algebras, the one-sided case is more subtle.
		We prove that these two notions remain equivalent for closed right ideals of an arbitrary $C^*$-algebra.
		Next, using the correspondence between submodules and right ideals of the algebra of compact operators, we extend this result to closed submodules of Hilbert $C^*$-modules.
		In the commutative case, where a Hilbert module can be realized as a continuous field of Hilbert spaces, we give a geometric reformulation of essentiality and derive a fiberwise criterion.
	\end{abstract}
	
	\maketitle
	
	\section{Introduction}
	\label{sec:intro}
	For an ideal of a $C^*$-algebra, there are several natural ways to say that it is large.
	Some of these notions, together with their natural generalizations to the case of submodules of Hilbert modules, were considered in \cite{manuilov}.
	It is well known that if one restricts attention to closed two-sided ideals, then all these properties are equivalent.
	For one-sided ideals, however, this equivalence fails, and that paper establishes some relations between different definitions of a ``large'' closed one-sided ideal.
	
	For general background on Hilbert $C^*$-modules, we refer the reader to \cite{lance}.
	We focus on two notions for closed submodules of Hilbert $C^*$-modules, namely essentiality and topological essentiality.
	\begin{definition}
		\label{def:ess}
		Let $\M$ be a Hilbert module over a $C^*$-algebra $A$.
		A closed submodule $\N \subset \M$ is called
		\begin{itemize}
			\item[$(E)$] \textit{essential} if for every nonzero (not necessarily closed) submodule $\mathcal K \subset \M$, the intersection $\N \cap \mathcal K$ is nontrivial;
			\item[$(TE)$] \textit{topologically essential} if for every nonzero closed submodule $\mathcal K \subset \M$, the intersection $\N \cap \mathcal K$ is nontrivial.
		\end{itemize}
		In the particular case where the algebra $A$ is regarded as a right Hilbert $C^*$-module over itself, this yields the definitions of topologically essential $(TE)$ and essential $(E)$ right ideals of a $C^*$-algebra.
	\end{definition}
	
	The property $(E)$ may seem somewhat unnatural in the context of $C^*$-algebras and Hilbert modules, since it involves nonclosed submodules or ideals.
	However, it admits a reformulation that looks less algebraic and is more convenient in practice (see Lemma~\ref{lem:reformulation}).
	Our main goal is to prove that these definitions are equivalent.
	
	In Section~\ref{sec:ideals} we show that essentiality and topological essentiality are equivalent for closed right ideals of an arbitrary $C^*$-algebra.
	
	In Section~\ref{sec:modules} the result of Section~\ref{sec:ideals} is generalized to submodules of Hilbert $C^*$-modules.
	
	In Section~\ref{sec:fields} we consider the important special case of a Hilbert module over a commutative $C^*$-algebra. In this setting, the module can be viewed as a continuous field of Hilbert spaces over a locally compact Hausdorff space, and essentiality of a submodule can be reformulated in geometric terms.
	
	\section{$(TE) \iff (E)$ for one-sided ideals of a C*-algebra}
	\label{sec:ideals}
	Clearly, $(E) \Rightarrow (TE)$. To prove the converse implication, we need several standard facts from the operator algebra theory.
	
	Let $A$ be an arbitrary $C^*$-algebra.
	By means of the universal representation (see \cite[Theorem I.9.12]{davidson}), we may identify it with a closed $*$-subalgebra of $\mathbb B(H)$, the algebra of bounded operators on a Hilbert space.
	Denote by $A''$ the bicommutant of $A$, that is, the smallest von Neumann subalgebra of $\mathbb B(H)$ containing $A$.
	A projection $p \in A''$ is called \textit{open} if there exists a net of positive elements $a_s \in A_{+}$ increasing monotonically to $p$ in the strong operator topology.
	It is known (see Theorem 1.1 in \cite{pedersen}) that closed one-sided ideals of $A$ admit a description in terms of open projections.
	More precisely, for every open projection $p \in A''$, the set $J = pA'' \cap A$ is a closed right ideal of $A$, and every closed right ideal is of that form.
	
	Every von Neumann algebra $M$ admits the Borel functional calculus for self-adjoint elements (see \cite[Corollary II.2.6]{davidson}).
	That is, for every $a = a^* \in M$, there is a $^*$-homomorphism from the algebra of bounded Borel functions on the spectrum $\sigma(a)$ to the algebra $M$, sending the function $g(t) = t$ to $a$.
	
	The following two technical lemmas are standard.
	\begin{lemma}
		Let $g_k$ be an increasing sequence of continuous functions on the spectrum of a self-adjoint element $a \in A$.
		Suppose that $g_k \to g$ pointwise as $k \to \infty$.
		Then the sequence $g_k(a)$ converges increasingly to $g(a) \in A''$ in the strong operator topology.
	\end{lemma}
	For a proof one may refer to Lemma I.6.4 in \cite{davidson}.
	\begin{lemma}
		\label{lem:proj}
		Let $a = a^* \in A$, $\eps > 0$, and let $\chi_{(\eps, +\infty)}$ be the characteristic function of the ray $t > \eps$.
		Then $p = \chi_{(\eps, +\infty)}(a)$ is an open projection in $A''$.
	\end{lemma}
	\begin{proof}
		The Borel function $g :=\chi_{(\eps, +\infty)}$ can be approximated by an increasing sequence of continuous functions on the real line $\RR$.
		For example,
		\[ g_n(t) = \begin{cases}
			0,\; t \leq \eps \\
			n(t - \eps),\; t \in (\eps, \eps + \frac 1n) \\
			1,\; t \in [\eps + \frac 1n, +\infty)
		\end{cases} \]
		Hence, by the previous lemma, the sequence $g_n(a)$ increases to $g(a)$ in the strong operator topology. Therefore $g(a) = p$ is an open projection.
	\end{proof}
	
	We now prove the equivalence of the properties $(TE)$ and $(E)$ for one-sided ideals of a $C^*$-algebra.
	\begin{theorem}
		\label{thm:main_ideals}
		Every topologically essential right ideal $J$ of a $C^*$-algebra $A$ is essential.
	\end{theorem}
	\begin{proof}
		Let $I$ be a nonzero right ideal of $A$, not necessarily closed. Choose $x \in I$, $x \ne 0$. Then $a = xx^* \ne 0$ is a positive element of $I$.
		Using the Borel functional calculus in $A''$, set
		\[ (a-\eps)_+ = \max(a-\eps, 0) = (a-\eps)\chi_{(\eps, +\infty)}(a). \]
		There exists $\eps > 0$ such that $(a-\eps)_+ \ne 0$.
		Indeed, otherwise for every $\eps > 0$ we would have $a-\eps \leq 0$, hence $a \leq 0$, which implies $a = 0$.
		
		Thus $p = \chi_{(\eps, +\infty)}(a)$ is a nonzero projection, and by Lemma~\ref{lem:proj} it is open. Hence $K=pA'' \cap A$ is a nonzero closed right ideal of $A$.
		
		Let $g$ be a continuous function on the real line $\RR$ such that $g(t) = 0$ for $t < \frac\eps 2$ and $g(t) = \frac1t$ for $t \geq \eps$.
		Put $f(t) = tg(t)$ and consider $f(a) = ag(a) \in I$. We have
		\[ f(a)p = f\chi_{(\eps, +\infty)}(a) = \chi_{(\eps, +\infty)}(a) = p, \]
		since $f(t) = 1$ for $t \geq \eps$.
		Now take any $b \in K$. Then
		\[ b = pb = f(a)pb \in I. \]
		Thus $K \subset I$ and by topological essentiality of the ideal $J$ we have
		\[ J \cap I \supset J \cap K \ne 0. \qedhere \]
	\end{proof}
	
	\begin{remark}
		The proof also yields the following interesting fact. Every nonzero right ideal $I \subset A$ contains a norm-closed right subideal $K \subset I$.
	\end{remark}
	
	\section{Generalization to Hilbert modules}
	\label{sec:modules}
	Let $\M$ be a right Hilbert $C^*$-module over a $C^*$-algebra $A$.
	Denote by $\KK(\M)$ the $C^*$-algebra of compact operators on $\M$, that is, the closure of the linear span of the elementary operators $\Theta_{x,y}(z) := x\inner yz$ over all $x, y \in \M$.
	
	It is well known (see \cite[Corollary 4.2]{meyer}) that closed submodules of $\M$ are in one-to-one correspondence with hereditary subalgebras, and hence with closed right ideals of $\KK(\M)$.
	Namely, for a closed submodule $\N \subset \M$ define the ideal $J_\N = \{ T \in \KK(\M) \mid \operatorname{Ran} T \subset \N \}$.
	Conversely, from a closed right ideal $J \subset \KK(\M)$ one recovers the submodule $\N = \overline{J\M}$.
	
	Theorem 6.1 of \cite{manuilov} states that this correspondence preserves essentiality, that is, $\N$ is an $(E)$-submodule if and only if $J_\N$ is an $(E)$-ideal.
	We now show that the same is true for topological essentiality.
	
	\begin{theorem}
		\label{thm:modules_ideals}
		Let $\N$ be a closed submodule of a Hilbert $C^*$-module $\M$.
		Then $\N$ is topologically essential if and only if $J_\N$ is topologically essential.
	\end{theorem}
	For the proof we need two standard technical lemmas.
	\begin{lemma}
		\label{lem:theta}
		The map $(x,y) \mapsto \Theta_{x,y}$ from $\M \times \M$ to $\KK(\M)$ is jointly continuous. Moreover, $\Theta_{x,x} = 0$ implies $x = 0$.
	\end{lemma}
	\begin{proof}
		The first statement follows from an easy computation:
		\begin{eqnarray*}
			(\Theta_{x,y} - \Theta_{x',y'})(z) &=& x\inner{y}{z} - x'\inner{y'}{z}  \\
			&=& x\inner{y}{z} - x\inner{y'}{z} + x\inner{y'}{z} - x'\inner{y'}{z}  \\
			&=& x\inner{y - y'}{z} + (x - x')\inner{y'}{z},
		\end{eqnarray*}
		whence $\|\Theta_{x,y} - \Theta_{x',y'}\|\leq \|x\|\|y-y'\| + \|x - x'\|\|y'\|$.
		
		Now suppose that $\Theta_{x,x} = 0$.
		Then $m := \Theta_{x,x}(x) = x\langle x,x \rangle = 0$, and hence
		\[
		\inner mm = \inner xx^3 = 0.
		\]
		Therefore $\inner{x}{x} = 0$, since $\inner{x}{x}$ is positive. Thus $x=0$.
	\end{proof}
	\begin{remark}
		The module $\M$ can be regarded as a left Hilbert $C^*$-module over $\KK(\M)$ with the inner product $\inner{x}{y} = \Theta_{x,y}$.
	\end{remark}
	
	The following reformulation of the properties $(E)$ and $(TE)$ is useful (see \cite{manuilov}).
	\begin{lemma}
		\label{lem:reformulation}
		A closed submodule $\N \subset \M$ is
		\begin{enumerate}
			\item[(1)] essential if and only if for every nonzero $m \in \M$ there exists an element $a \in A$ such that $ma \in \N$ and $ma \ne 0$;
			\item[(2)] topologically essential if and only if for every nonzero $m \in \M$ there exists a sequence $a_k \in A$ such that $ma_k$ converges to some nonzero $n \in \N$.
		\end{enumerate}
	\end{lemma}
	\begin{proof}
		Note that these conditions are equivalent to
		\begin{align}
			\forall m \ne 0\quad \N \cap mA \ne 0 \\
			\forall m \ne 0\quad \N \cap \overline{mA} \ne 0
		\end{align}
		respectively.
		Clearly, if $m \in \M$ and $m \ne 0$, then the submodule $mA$, and therefore also $\overline{mA}$, is nontrivial, since $m\inner{m}{m} \ne 0$.
		Thus the implications from left to right are obvious.
		
		Now let $\mathcal K \subset \M$ be a nonzero submodule.
		Take $m \in \mathcal K \setminus \{0\}$; then $mA \subset \mathcal K$ is nontrivial.
		If, in addition, $\mathcal K$ is closed, then $\overline{mA} \subset \mathcal K$. Hence, in the two cases,
		\begin{gather*}
			\N \cap \mathcal K \supset \N \cap mA \ne 0 \\
			\N \cap \mathcal K \supset \N \cap \overline{mA} \ne 0 \hfill \qedhere
		\end{gather*}
	\end{proof}
	We now proceed to the proof of the theorem.
	\begin{proof}[Proof of Theorem~\ref{thm:modules_ideals}]
		First assume that $\N$ is topologically essential.
		Take any nonzero $T \in \KK(\M)$ and choose $m \in \M$ such that $Tm \ne 0$.
		By Lemma~\ref{lem:reformulation}, there exists a sequence $a_k \in A$ such that $Tma_k \to n \in \N \setminus \{0\}$.
		Therefore, by Lemma~\ref{lem:theta},
		\[
		T\Theta_{ma_k,Tma_k} = \Theta_{Tma_k,Tma_k} \to \Theta_{n,n} \in J_\N \setminus \{0\},
		\]
		which proves topological essentiality of $J_\N$.
		
		Conversely, let $J_\N$ be topologically essential. Take a nonzero $m \in \M$.
		There exists a sequence $S_k \in \KK(\M)$ such that $\Theta_{m,m}S_k\rightarrow T$ for some $T\in J_\N \setminus \{0\}$.
		Suppose that $Tx = n \ne 0$; then for the sequence $a_k = \inner{m}{S_kx}$ we have $ma_k = \Theta_{m,m}S_kx$ converging to $Tx = n$, so $\N$ is topologically essential.
	\end{proof}
	
	Thus the equivalence of the properties $(TE)$ and $(E)$ for submodules of Hilbert $C^*$-modules reduces to their equivalence for one-sided ideals of $C^*$-algebras, and we obtain the following corollary.
	\begin{corollary}
		\label{thm:main_modules}
		Let $A$ be an arbitrary $C^*$-algebra.
		Then every topologically essential submodule $\N$ of a Hilbert $A$-module $\M$ is essential. 
	\end{corollary}
	
	\section{Essential submodules of $C_0(X)$-modules}
	\label{sec:fields}
	We now give another, simpler proof of the equivalence of the properties $(E)$ and $(TE)$ for submodules of Hilbert $C^*$-modules over commutative $C^*$-algebras.
	\begin{proposition}
		Let $A$ be a commutative $C^*$-algebra.
		Then every topologically essential submodule $\N$ of a Hilbert $A$-module $\M$ is essential.
	\end{proposition}
	\begin{proof}
		Let $\N \subset \M$ be a topologically essential submodule, and let $m \in \M$, $m \ne 0$, be arbitrary.
		We show that the submodule $\N \cap mA$ is nontrivial.
		By Lemma~\ref{lem:reformulation}, one can find a sequence $a_k \in A$ and a nonzero $n \in \N$ such that $ma_k \to n$ as $k \to \infty$.
		By continuity of the inner product, we obtain
		\[ \inner{n}{m}a_k = \inner{n}{ma_k} \to \inner{n}{n}. \]
		Since the module action $\M \times A \to \M$ is continuous, we also have
		\[ m\inner{n}{m}a_k \to m\inner{n}{n}. \]
		On the other hand, since $A$ is commutative,
		\[ m\inner{n}{m}a_k = ma_k\inner{n}{m} \to n\inner{n}{m}. \]
		Hence $m\inner{n}{n} = n\inner{n}{m}$. It remains to note that $n' := n\inner{n}{m}$ is a nonzero element of the submodule $\N$, since 
		\[ n'a_k = n\inner{n}{ma_k} \to n\inner{n}{n} \ne 0. \]
		Finally, for $a' = \inner{n}{n}$ we have $ma' = n'$, as required.
	\end{proof}
	
	It is well known (see \cite{takahashi}) that every Hilbert $C^*$-module $\M$ over a commutative $C^*$-algebra $A$ can be identified with the module $C_0(X,\mathcal H)$ of continuous sections of a continuous field of Hilbert spaces $\mathcal H=\{H_x\}_{x\in X}$ over the locally compact Hausdorff space $X=\operatorname{Sp}(A)$.
	Moreover, in the one-point compactification of $X$, the fiber over $\infty$ is the zero space $H_\infty=0$, that is, sections vanish at infinity.
	
	It is also known that every closed submodule $\N \subset C_0(X, \mathcal H)$ is of the form
	\[ \N = \{m \in \M \mid \forall x \in X\; m(x) \in L_x\} = C_0(X, \mathcal L), \]
	where $\mathcal L = \{L_x\}_{x \in X}$ is a continuous field of closed subspaces of $\mathcal H$.
	Namely, 
	\[ L_x = \overline{ \{ n(x) \mid n \in \N \} }. \]
	
	In these terms one obtains the following general criterion for essentiality of a submodule of a $C_0(X)$-module.
	
	\begin{proposition}
		\label{prop:criterion}
		Let $\mathcal H = \{H_x\}_{x \in X}$ be a continuous field of Hilbert spaces over $X$, vanishing at infinity, and let $\M = C_0(X, \mathcal H)$ be the module of its continuous sections.
		Let $\mathcal L = \{L_x\}_{x \in X}$ be the field of subspaces corresponding to a closed submodule $\N \subset \M$.
		Then $\N$ is an essential submodule of $\M$ if and only if for every $m \in \M$ the set
		\[ Y_m = \{ x \in X \mid m(x) \notin L_x \} \]
		is nowhere dense.
	\end{proposition}
	\begin{proof}
		Let $m \in \M$ be a nonzero element and assume that $Y_m$ is nowhere dense. Define
		\[
		Z_m = \{x \in X \mid m(x) \ne 0\}.
		\]
		Then $Y_m \subset Z_m$, and $Z_m$ is open. Since $Y_m$ is nowhere dense, the set $Z_m \setminus \overline{Y_m}$
		is a nonempty open subset of $X$.
		Choose a nonzero function $a \in C_0(X)$ such that $\operatorname{supp} a \subset Z_m \setminus \overline{Y_m}$.
		Then $a$ vanishes on $Y_m$, so $ma \in \N$, while $ma \ne 0$. Thus nowhere denseness of all sets $Y_m$ is sufficient for essentiality of the submodule.
		
		Conversely, suppose that for some $m \in \M$ the open set $V = \operatorname{Int}\overline{Y_m}$ is nonempty.
		Take $a \in C_0(X)$ with $\emptyset \ne \operatorname{supp} a \subset V$ and put
		\[
		U = \{x \in X \mid a(x) \ne 0\}.
		\]
		Consider $ma \in \M$. Clearly
		\[
		Z_{ma} = Z_m \cap U \text{ and } Y_{ma} = Y_m \cap U.
		\]
		Since $U \ne \emptyset$ is open and $U \subset \overline{Y_m} \subset \overline{Z_m}$, we obtain
		\[
		\overline{Z_{ma}} = \overline{\overline{Z_m} \cap U} = \overline U = \overline{\overline{Y_m} \cap U} = \overline{Y_{ma}}.
		\]
		Thus $ma \ne 0$ and $Y_{ma}$ is dense in $Z_{ma}$.
		Now let $b \in C_0(X)$ and assume that $mab \in \N$.
		Then $b(x)=0$ for all $x \in Y_{ma}$, hence $b \equiv 0$ on $Z_{ma}$, and therefore $mab=0$.
		Thus $\N$ is not essential.
	\end{proof}
	
	One would like to describe essentiality of a submodule in terms of the set
	\[ Y := \{ x \in X \mid L_x \ne H_x \} = \bigcup_{m \in \M}Y_m. \]
	However, the following example shows that in general this is impossible.
	
	\begin{example}
		Let $X$ be a nonseparable compact space (for instance, $[0,1]^\RR$), and let $H = \ell_2(X)$ be a nonseparable Hilbert space with an orthonormal basis $(e_t)_{t \in X}$ indexed by $X$.
		For $x \in X$, set $H_x = H$ and $L_x = e_x^\bot = \overline{\mathrm{span}}\{e_t \mid t \ne x\}$.
		Then $Y = X$, but $\N = C(X, \mathcal L)$ is essential in $\M = C(X, \mathcal H)$.
	\end{example}
	\begin{proof}
		Observe that since all spaces $H_x$ are equal to the same $H$, the module of sections $C(X, \mathcal H)$ can be identified with the module of all continuous maps $X \to H$.
		Under this identification, the image $m(X)$ of each $m \in \M$ is compact in the Hilbert space $H$ and therefore lies in a separable subspace $\overline{\mathrm{span}}\{e_t \mid t \in J\}$ for some countable set $J$.
		
		Hence for every $x \in X$, in the Fourier expansion $m(x) = \sum_t e_t a_t(x)$, all terms with $t \notin J$ vanish.
		Therefore
		\[
		Y_m = \{x \in X \mid m(x) \notin e_x^\bot\} = \{x \in X \mid a_x(x) \ne 0\} \subset J,
		\]
		so $Y_m$ is countable and hence nowhere dense in $X$.
	\end{proof}
	
	We now show that in the separable case such difficulties do not arise, and that essentiality of a submodule is in fact equivalent to the nowhere denseness of the set $Y$.
	\begin{theorem}
		\label{thm:criterion}
		Let $X$ be a locally compact Hausdorff space, let $H_x$ be a continuous field of Hilbert spaces over $X$, and let $\M$ be the module of its continuous sections vanishing at infinity.
		Assume that $X$ is separable and that $\M$ is countably generated as a Hilbert $C_0(X)$-module.
		Let $\N \subset \M$ be a closed submodule, and let $L_x$ be the corresponding continuous field of closed subspaces.
		Then $\N$ is essential if and only if the set 
		\[Y = \{x \in X \mid L_x \ne H_x\}\]
		is nowhere dense.
	\end{theorem}
	\begin{proof}
		If $Y$ is nowhere dense, then for every $m \in \M$ we have $Y_m \subset Y$, hence $Y_m$ is nowhere dense, and therefore $\N$ is essential by Proposition~\ref{prop:criterion}.
		
		Conversely, assume that $\N$ is essential, but $Y$ is not nowhere dense. Then some nonempty open set $U$ is contained in $\overline Y$.
		Hence $\overline{U \cap Y} = \overline{U \cap \overline Y} = \overline U$, and by separability of $X$ one can choose a countable set $\{x_j\}_{j \in \NN} \subset U \cap Y$ dense in $U$.
		
		Take a countable generating set $\{g_k\}_{k \in \NN} \subset \M$ with $\|g_k\| = 1$ for all natural $k$.
		For convenience, write $Y_k = Y_{g_k},\; H_j = H_{x_j}$ and $L_j = L_{x_j} \subset H_j$.
		Note that
		\begin{equation}
			\label{eq:union}
			Y = \bigcup_{k=1}^\infty Y_k.
		\end{equation}
		Indeed, if $x \in Y_k$ for some $k$, then clearly $L_x \ne H_x$, so $x \in Y$.
		Conversely, if for all $k \in \NN$ one has $g_k(x) \in L_x$, then for every $m \in \M$ one also has $m(x) \in L_x$.
		But the set $\{m(x) \mid m \in \M\}$ is dense in $H_x$, hence $L_x = H_x$ and therefore $x \notin Y$.
		
		Now fix $j \in \NN$. Since $x_j \in Y$, it follows from \eqref{eq:union} that one can choose an index $k_j$ such that $x_j \in Y_{k_j}$, that is, $g_{k_j}(x_j) \notin L_j$.
		Since $X$ is completely regular, there exist functions $\{a_j\}_{j \in \NN} \subset C_0(X)$ with $0 \leq a_j \leq 1$, $a_j(x_j) = 1$, and $a_j(x_i) = 0$ for all $i < j$.
		Construct inductively a sequence of positive numbers $0 < \lambda_j \leq 2^{-j}$ such that
		\begin{equation}
			\label{eq:induction}
			\sum_{i=1}^j\lambda_ia_i(x_j)g_{k_i}(x_j) = \sum_{i=1}^{j-1}\lambda_ia_i(x_j)g_{k_i}(x_j) + \lambda_jg_{k_j}(x_j) \notin L_j. 
		\end{equation}
		This is possible because the image of the vector $g_{k_j}(x_j)$ in the quotient $H_j / L_j$ is nonzero.
		
		Now put
		\[ m = \sum_{j=1}^\infty \lambda_j g_{k_j}a_j. \]
		Then $m \in \M$, since the norm of the $j$-th term does not exceed $2^{-j}$, and hence the series converges uniformly.
		By construction, from~\eqref{eq:induction} it follows that for every natural $j$,
		\[ m(x_j) = \sum_{i=1}^j\lambda_ig_{k_i}(x_j)a_i(x_j) \notin L_j. \]
		Thus all $x_j$ belong to $Y_m$, and therefore $U \subset \overline{Y_m}$.
		Hence the set $Y_m$ is not nowhere dense, and the submodule $\N$ is not essential, a contradiction.
	\end{proof}
	
	\begin{remark}    
		Every continuous field of Hilbert spaces $\{H_x\}_{x \in X}$ gives rise to a continuous field of $C^*$-algebras $\{\KK(H_x)\}_{x \in X}$ consisting of compact operators on the fibers.
		It is easy to verify that the $C^*$-algebra $C_0(X, \KK(H_x))$ of continuous sections of this field coincides with $\KK(C_0(X, H_x))$, the algebra of compact operators on $\M = C_0(X, \mathcal H)$.
		As was noted above, all closed right ideals of $\KK(\M)$ are of the form
		\[ J = \{T \in \KK(\M) \mid \operatorname{Ran} T_x \subset L_x\} \]
		for some field of closed subspaces $L_x \subset H_x$.
		Then, taking into account the results of Section~\ref{sec:modules}, criteria analogous to~\ref{prop:criterion} and~\ref{thm:criterion} are also valid for closed right ideals of the $C^*$-algebra of compact operators on the module $\M$.
	\end{remark}

\end{document}